\documentclass[12pt,a4paper,oneside,english]{amsart}
\usepackage[T1]{fontenc}
\usepackage[latin9]{inputenc}
\usepackage{units}
\usepackage{amsthm}
\usepackage{amssymb}
\usepackage{bbm}
\usepackage{amsmath}
\usepackage{xcolor}

\makeatletter

\numberwithin{equation}{section}
\numberwithin{figure}{section}
\theoremstyle{plain}
\newtheorem{thm}{\protect\theoremname}[section]
  \theoremstyle{plain}
  
  \theoremstyle{plain}
  
  \theoremstyle{remark}
  
  \theoremstyle{plain}

\@ifundefined{date}{}{\date{}}

\usepackage{amsthm}

\usepackage{cite}

\newtheorem{theorem}{Theorem}[section]

\newtheorem{corollary}[theorem]{Corollary}

\newtheorem{example}[theorem]{Example}

\newtheorem{lemma}[theorem]{Lemma}

\newtheorem{proposition}[theorem]{Proposition}
\newtheorem{remark}[theorem]{Remark}

\providecommand{\lemmaname}{Lemma}
  \providecommand{\propositionname}{Proposition}
  \providecommand{\remarkname}{Remark}
\providecommand{\theoremname}{Theorem}

\usepackage{babel}
\providecommand{\lemmaname}{Lemma}
  \providecommand{\propositionname}{Proposition}
  \providecommand{\remarkname}{Remark}
\providecommand{\theoremname}{Theorem}

\usepackage{babel}
\providecommand{\corollaryname}{Corollary}
  \providecommand{\lemmaname}{Lemma}
  \providecommand{\propositionname}{Proposition}
  \providecommand{\remarkname}{Remark}
\providecommand{\theoremname}{Theorem}

\makeatother

\usepackage{babel}
  \providecommand{\corollaryname}{Corollary}
  \providecommand{\lemmaname}{Lemma}
  \providecommand{\propositionname}{Proposition}
  \providecommand{\remarkname}{Remark}
\providecommand{\theoremname}{Theorem}

\begin{document}

\title[Explicit models of $\ell_1$-preduals and the $w^*$-FPP in $\ell_1$]{Explicit models of $\ell_1$-preduals and \\ the weak$^*$ fixed point property in $\ell_1$}

\dedicatory{Dedicated to the Memory of Professor Kazimierz Goebel}

\author{Emanuele Casini}

\address{Dipartimento di Scienza e Alta Tecnologia, Università dell'Insubria,
via Valleggio 11, 22100 Como, Italy }

\email{emanuele.casini@uninsubria.it}

\author{Enrico Miglierina}

\address{Dipartimento di Matematica per le Scienze Economiche, Finanziarie ed Attuariali,
Università Cattolica del Sacro Cuore, Via Necchi 9, 20123 Milano,
Italy }

\email{enrico.miglierina@unicatt.it}

\author{\L ukasz Piasecki}

\address{Instytut Matematyki,
Uniwersytet Marii Curie-Sk{\l}odowskiej, Pl. Marii Curie-Sk{\l}odowskiej 1, 20-031 Lublin, Poland}

\email{piasecki@hektor.umcs.lublin.pl}

\begin{abstract}
We provide a concrete isometric description of all the preduals of $\ell_1$ for which the standard basis in $\ell_1$ has a finite number of $w^*$-limit points. Then, we apply this result to give an example of an $\ell_1$-predual $X$ such that its dual $X^*$ lacks the weak$^*$ fixed point property for nonexpansive mappings (briefly, $w^*$-FPP), but $X$ does not contain an isometric copy of any  hyperplane $W_{\alpha}$ of the space $c$ of convergent sequences such that $W_\alpha$ is a predual of $\ell_1$ and $W_\alpha^*$ lacks the $w^*$-FPP. This answers a question left open in the 2017 paper of the present authors.
\end{abstract}



\maketitle

\section{Introduction, notations and preliminaries}

The study of Lindenstrauss spaces, i.e. the Banach spaces $X$ whose duals $X^*$ are isometric to an $L_1(\mu)$ space for some measure $\mu$, received a lot of attention between 60's and 70's of the last century. In that period, deep results on classification and structural properties were achieved (see, e.g., \cite{Lindenstrauss 1964}, \cite{Lazar-Lindenstrauss1971}, \cite{Lindenstrauss-Wulbert} and the monograph \cite{Lacey book}). Among these results we recall the paper \cite{Zippin1969} where Zippin proved that every infinite-dimensional Lindenstrauss space contains an isometric copy of the space $c_0$ of sequences converging to zero, and the work \cite{Semadeni1964} where Semadeni proved that if the closed unit ball of a Lindenstrauss space has at least one extreme point, then it is isometric to the space $A(S)$ of affine continuous functions on a Choquet simplex $S$. It is worth to mention that Lindenstrauss spaces found a natural application in the problem of the extension of compact operators (see \cite{Lindenstrauss 1964}, \cite{Lazar1969}) and they are strictly connected to various geometrical properties of Banach spaces such as polyhedrality (a Banach space is called polyhedral whenever all finite-dimensional sections of its closed unit ball are polytopes - see, e.g., \cite{Durier & Papini}, \cite{Fonf Lindenstrauss & Phelps 2001}, \cite{Fonf & Vesely} and the reference therein). However, recent studies have shown that the knowledge of the structure of Lindestrauss spaces does not appear to be complete. Indeed, the authors of the present paper and L. Vesel\'y proved in \cite{Casini-Miglierina-Piasecki-Vesely2016} the existence of a ``strange'' Lindenstrauss space that does not fit well in the framework of the previous classification (see also \cite{Zippin 2018}). The idea behind the construction of the example of such a space is based on the previous papers \cite{Casini-Miglierina-Piasecki2014} and \cite{Casini-Miglierina-Piasecki2015} by the present authors. In \cite{Casini-Miglierina-Piasecki2014} we gave a complete classification of closed hyperplanes in the space $c$ of convergent sequences, which is a typical example of a Lindenstrauss space, since it is well known that $c^*=\ell_1$. 
The relevance of the role played by the closed hyperplanes of $c$ is underlined also by Theorem 1 in \cite{Zippin 2018} that shows how these hyperplanes are ubiquitous in the separable Lindenstrauss spaces.   The main feature of interest of the results in \cite{Casini-Miglierina-Piasecki2014}, beside the example mentioned above, is that they provide a complete isometric description of preduals of $\ell_1$ when it is assumed that the standard basis of $\ell_1$ is $w^*$-convergent. More precisely, if the standard basis $\{e^*_n\}$ of $\ell_1$ is $\sigma(\ell_1,X)$-convergent to $\alpha$ ($e^*_n \xrightarrow {\sigma(\ell_1,X)} \alpha$), then there exists a unique (up to isometry) hyperplane $W_{\alpha}$ in $c$ such that $X=W_{\alpha}$. This classification result helped us to conclude a relationship between a certain property describing the inner structure of a separable Lindenstrauss space and a fixed point property for its dual. It is worth mentioning that the metric fixed point theory is currently one of the basic branches of nonlinear analysis, finding numerous applications e.g. in the theory of differential and integral equations, in the theory of holomorphic mappings, in the theory of fractals and in graph theory. We refer a reader not acquainted with the metric fixed point theory to the books by Goebel and Kirk \cite{GoKi} and Goebel \cite{Go1}. 

We recall that a nonempty, bounded, closed and convex subset $C$ of $X$ has the fixed point property
(shortly, FPP) if each nonexpansive mapping (i.e., the mapping $T:C\rightarrow C$ such that $\|T(x)-T(y)\|\leq \|x-y\|$ for all $x,y\in C$) has a fixed point. If every such set $C$ has the FPP, then we say that $X$ has the FPP.
The dual space $X^*$ is said to have the $\sigma(X^*,X)$-fixed
point property (shortly, $\sigma(X^*,X)$-FPP) if every nonempty, convex,
$\sigma(X^*,X)$-compact subset $C$ of $X^*$ has the FPP. The study of the
$\sigma(X^*,X)$-FPP reveals to be of special interest whenever a
dual space has different preduals. Indeed, the behaviour with
respect to the $\sigma(X^*,X)$-FPP of a given dual space can be
completely different if we consider two different preduals. For
instance, this situation occurs when we consider the space $\ell_1$
and its preduals $c_0$ and $c$ where it is well-known (see
\cite{Karlovitz1976}) that $\ell_1$ has the $\sigma(\ell_1,c_0)$-FPP
whereas it lacks the $\sigma(\ell_1,c)$-FPP.
As we said above, the result proved in \cite{Casini-Miglierina-Piasecki2014} helped us to deduce a necessary and sufficient condition for the $\sigma(\ell_1,X)$-FPP in terms of the inner structure of an $\ell_1$-predual $X$. Since this result plays a central role in our paper, we need to introduce some notations to make readable its statement. Moreover, these notations will be used throughout the whole paper.
We recall that the standard duality between $c$ and $\ell_1$ is defined by
$$
x^*(x)=x^{*}(1)\lim_i x(i) + \sum_{i=1}^{\infty}x^{*}(i+1)x(i)
$$
for every $x=(x(1),x(2),\dots)\in c$ and every $x^*=(x^*(1),x^*(2), \dots)\in \ell_1$. 

For every $\alpha=(\alpha(1),\alpha(2),\dots)\in \ell_1$ with $\|\alpha\|\leq 1$, we define the following hyperplane in $c$:
$$
W_\alpha=\left\lbrace x=(x(1),x(2),\dots)\in c: \lim_{i \rightarrow \infty}x(i)=\sum_{i=1}^{\infty}\alpha(i)x(i) \right\rbrace.
$$
For a detailed study of this class of spaces we refer the reader to \cite{Casini-Miglierina-Piasecki2014} and \cite{Casini-Miglierina-Piasecki2015}. In \cite{Casini-Miglierina-Piasecki2014} we proved that $W^*_\alpha=\ell_1$ and $e^*_n \xrightarrow {\sigma(\ell_1,W_\alpha)} \alpha$. Moreover, since $\ell_1$ lacks the $\sigma(\ell_1,W_\alpha)$-FPP if and only if $\left\|\alpha \right\|=1$ and $\alpha$ has infinitely many non-negative components, we call ``bad''  any hyperplane $W_\alpha$ with $\alpha$ satysfying these properties.

Now, we can recall the necessary and sufficient condition for the $\sigma(\ell_1,X)$-FPP mentioned above. This theorem seems to be very interesting since it is very hard to find a complete characterization of fixed point property for a class of Banach spaces. Indeed, to our best knowledge, in the literature the only result in this vein is the theorem asserting that a subspace of $L_1([0,1])$ has the FPP if and only if it is reflexive (see \cite{Maurey1980-1981}, \cite{Dowling-Lennard1997}).

\begin{theorem}\label{OGT} [Theorem 4.1 in \cite{Casini-Miglierina-Piasecki2015}]
	Let $X$ be a predual of $\ell_1$. Then the following are equivalent.
\begin{enumerate}
	\item \label{item noFPP} $\ell_1$ lacks the $\sigma(\ell_1,X)$-FPP.
	
	\item \label{item subsequence}There is a subsequence $\{e_{n_k}^*\}_{k \in
		\mathbb{N}}$ of the standard basis $\{e_{n}^*\}_{n \in \mathbb{N}}$ in
	$\ell_1$ which is $\sigma(\ell_1,X)$-convergent to a norm-one
	element $e^* \in \ell _1$ with $e^*(n_k)\geq 0$ for all $k\in
	\mathbb{N}$.
	
	\item\label{item quotient isometric} There is a quotient of $X$ isometric to a ``bad'' $W_\alpha$.
	
	\item \label{item quotient contained} There is a quotient of $X$ that contains a subspace isometric to a ``bad'' $W_\beta$.
	\end{enumerate}
\end{theorem}

This characterization involves the presence of a particular quotient in the predual $X$ of $\ell_1$. It is worth mentioning that the ``bad'' $W_{\alpha}$ and $W_{\beta}$ in statements (3) and (4) of Theorem \ref{OGT} cannot be replaced by $c$ (see Example 2.4 in \cite{Casini-Miglierina-Piasecki2015}). 
One of the most interesting feature of this theorem is the role played by the quotients of the predual space $X$. Indeed, to our knowledge, till now quotients have not a relevant role in the study of FPP. Moreover, in Theorem 3.7 of \cite{Casini-Miglierina-Piasecki2015} we proved  that the presence of an isometric copy of a ``bad'' $W_\alpha$ in a separable Banach space $X$ implies the lack of the $\sigma(X^*,X)$-FPP. Nevertheless, in \cite{Casini-Miglierina-Piasecki2015} we left unanswered the following open problem: does the lack of the $\sigma(\ell_1,X)$-FPP imply that $X$ contains an isometric copy of a ``bad'' $W_\alpha$? In other words, can we replace ``quotient'' by  ``subspace'' in statement (3) of Theorem \ref{OGT}?

 The main aim of the present paper is to answer this question in negative  by showing that there exists a predual $X$ of $\ell_1$ such that $\ell_1$ lacks the $\sigma(\ell_1,X)$-FPP but $X$ does not contain an isometric copy of any ``bad'' $W_\alpha$ (Example \ref{Example}). To build this space we provide a concrete isometric representation for all the preduals $X$ of $\ell_1$ such that the set of $\sigma(\ell_1,X)$-limit points of the standard basis of $\ell_1$ contains a finite number of elements (Theorem \ref{l_1-preduals}, Corollary \ref{cor1} and Corollary \ref{cor2}). Finally, we mention that we also find a sufficient condition for the presence of an isometric copy of a hyperplane $W_\alpha$ in a predual of $\ell_1$ (Proposition \ref{Presence of W_e^*}).

Now, for the convenience of the reader, we recall some notations that we will use often throughout the paper. If $X$ is a Banach space, we denote by $B_X, S_X, X^*$ respectively, its closed unit ball, its unit sphere and its dual space. Finally, if $A\subset X^*$, $A'$ denotes the set of weak$^*$ limit points of $A$; moreover, if $A \subset X$, then we denote by $\mathrm{ext}\,A$ the set of its extreme points.

\section{A concrete representation of preduals of $\ell_1$ when the standard basis of $\ell_1$ has a finite number of $w^*$-limit points} \label{Sec Concrete repres}
In the sequel we will use the following notations to represent the ordinal
interval $\left[1,\omega n\right]$ (endowed with the order topology), where $\omega$ denotes the first infinite ordinal,
by means of the subset of $\mathbb{R}^{2}$ (with euclidean topology) given by 
$$
\Omega^n=\Omega_0^n 
\cup 
\left\{ \left( \frac{1}{i},0\right) \in\mathbb{R}^{2}:i=1,\dots,n\right\}
$$
where
$$
\Omega_{0}^n=\left\{ \left( \frac{1}{i},\frac{1}{j}\right)  \in\mathbb{R}^{2}:i=1,\dots,n, \, j=1,2,\dots\right\}. 
$$

It is well known that the dual of the space $C(\Omega^n)$ of all real valued continuous functions on the compact $\Omega^n$ is isometric to $\ell_{1}(\Omega^n)$.

The standard basis of $\ell_1(\Omega^n)$ is denoted by $\left\lbrace e^*_{i,j}\right\rbrace$ where $i=1,\dots,n$ and $j=0,1,\dots$ and, for every $f\in C(\Omega^n)$:
$$
e^*_{i,j}(f)=\left\lbrace \begin{array}{cc}
f\left( \frac{1}{i},\frac{1}{j}\right) & \mathrm{if }\, j\neq 0\\
f\left( \frac{1}{i},0\right) & \mathrm{if }\, j= 0
\end{array}\right. .
$$
Therefore a generic element of $\ell_1(\Omega^n)$ can be written as 
$$
y^*=\sum_{i=1}^n \left( \sum_{j=0}^{\infty} y^*(i,j)e^*_{i,j}\right) ,
$$
and, for every $f \in C(\Omega^n)$, we have
$$
y^*(f)=\sum_{i=1}^n \left( \sum_{j=1}^{\infty} f\left(\frac{1}{i},\frac{1}{j} \right) y^*(i,j)\right) +\sum_{i=1}^{n}f\left(\frac{1}{i},0 \right)y^*(i,0). 
$$
Similarly, an element $z^*\in \ell_1(\Omega_{0}^n)$ can be written as
$$
z^*=\sum_{i=1}^n \left( \sum_{j=1}^{\infty} z^*(i,j)e^*_{i,j}\right) .
$$
Now, let us fix $x_1^*,\dots,x_n^* \in B_{\ell_1(\Omega_{0}^n)}$, then we consider the subspace $W_{x_1^*,\dots,x_n^*}$ of $C(\Omega^n)$ defined by
$$
W_{x_1^*,\dots,x_n^*}=\left\lbrace f\in C(\Omega^n):f\left(\frac{1}{s},0 \right)=\sum_{i=1}^n \left( \sum_{j=1}^{\infty} f\left(\frac{1}{i},\frac{1}{j} \right)x_s^*(i,j)\right) \,\, \forall s=1,\dots,n  \right\rbrace.
$$

We can now state the main result of this section.

\begin{theorem} \label{l_1-preduals}
	The dual of the space $ W_{x_1^*,\dots,x_n^*}$ is isometric to $\ell_1$ (more precisely, to $\ell_1(\Omega_{0}^n)$) and the duality is given by the linear map $\phi: \ell_1\left( \Omega_0^n\right) \rightarrow W^*_{x_1^*,\dots,x_n^*}$ defined by
	$$
	\phi\left( z^*\right) \left( f\right)=\sum_{i=1}^n \left( \sum_{j=1}^{\infty} f\left(\frac{1}{i},\frac{1}{j} \right)z^*(i,j)\right) 
	$$
	for every $z^*\in \ell_1\left( \Omega_0^n\right)$ and every $f \in W_{x_1^*,\dots,x_n^*}$.
	
	Moreover, the $\sigma(\ell_1(\Omega_0^n),W_{x_1^*,\dots,x_n^*})$-limit points of the standard basis $\left\lbrace e^*_{i,j}\right\rbrace$ are $x_1^*,\dots,x_n^*$ and, for all $s=1,\dots,n$,
	$$
	e^*_{s,j} \xrightarrow[j\rightarrow \infty ]{\sigma(\ell_1(\Omega_0^n),W_{x_1^*,\dots,x_n^*})} x^*_s.
	$$  
\end{theorem}

\begin{proof}

	First we prove that $\phi$ is a surjective map. It is well known that $W^*_{x_1^*,\dots,x_n^*}$ can be identified with  
	$$
	\nicefrac{C(\Omega^n)^*}{W_{x_1^*,\dots,x_n^*}^\perp}=\nicefrac{\ell_1(\Omega^n)}{W_{x_1^*,\dots,x_n^*}^\perp}. 
	$$
	Now, let $h^*$ be a generic element of $ W^*_{x_1^*,\dots,x_n^*}$. Then, there exists $y^*\in \ell_1(\Omega^n)$ such that $h^*=y^*+W_{x_1^*,\dots,x_n^*}^\perp$ and
	$$
	y^*=\sum_{i=1}^n \left( \sum_{j=0}^{\infty} y^*(i,j)e^*_{i,j}\right).
	$$
	By the definition of the subspace $W_{x_1^*,\dots,x_n^*}$ we have that
	$$
	e^*_{s,0}(f)=f\left( \frac{1}{s},0\right)=\sum_{i=1}^n \left(  \sum_{j=1}^{\infty} f\left(\frac{1}{i},\frac{1}{j} \right)x_s^*(i,j)\right) =\sum_{i=1}^n \left(  \sum_{j=1}^{\infty} e_{i,j}^{*}(f)x_s^*(i,j)\right)
	$$
	for every $s=1,\dots,n$ and every $f\in W_{x_1^*,\dots,x_n^*}$.
	Therefore, if we consider $z^*\in \ell_1(\Omega_0^n)$ defined by
	$$
	z^*=\sum_{i=1}^n \left( \sum_{j=1}^{\infty} y^*(i,j)e^*_{i,j}\right) +\sum_{s=1}^{n}\left(  y^*(s,0)\left( \sum_{i=1}^n \left( \sum_{j=1}^{\infty} x_s^*(i,j)e^*_{i,j}\right) \right)\right) ,
	$$
	we conclude that $\phi(z^*)=h^*$.
	
	We now prove that the map $\phi$ is an isometry. It is easy to see that
	$$
	\left\| \phi(z^*)\right\|_{W^*_{x_1^*,\dots,x_n^*}}\leq \left\|z^* \right\|_{\ell_1(\Omega_0^n)}
	$$
	for every $z^* \in \ell_1(\Omega_0^n)$.
	On the other hand, let $z^*$ be a given element in $\ell_1(\Omega_0^n)$.
	Let $\left\lbrace f_N \right\rbrace_{N=1}^\infty \subset W_{x_1^*,\dots,x^*_n}$ be a sequence of functions defined, for all $i=1,\dots,n$, by
	$$
	f_N \left( \frac{1}{i},\frac{1}{j}\right)=\left\lbrace 
	\begin{array}{cc}
		 \mathrm{sgn} \left( z^*(i,j)\right)  &  \forall j=1,\dots,N \\
		g_i^N &  \forall j\geq N+1  
	\end{array}
	\right. 
	$$
	and $f_N\left( \frac{1}{i},0\right)=g_i^N$. 
	Since $f_N \in W_{x_1^*,\dots,x_n^*}$, the values $g_1^N,\dots,g_n^N$ should satisfy the following linear system of equations:
	\begin{equation*} 
	D^N g^N=A^N
	\end{equation*}
	where 
	$$
	D^N=
	\left[ 
	\begin{array}{cccc}
	1-\sum_{j=N+1}^{\infty} x^*_1(1,j)&-\sum_{j=N+1}^{\infty} x^*_1(2,j)&\cdots&-\sum_{j=N+1}^{\infty} x^*_1(n,j)\\
	-\sum_{j=N+1}^{\infty} x^*_2(1,j)&1-\sum_{j=N+1}^{\infty} x^*_2(2,j)&\cdots&-\sum_{j=N+1}^{\infty} x^*_2(n,j)\\
	\vdots&\vdots&\ddots&\vdots\\
	-\sum_{j=N+1}^{\infty} x^*_n(1,j)&-\sum_{j=N+1}^{\infty} x^*_n(2,j)&\cdots&1-\sum_{j=N+1}^{\infty} x^*_n(n,j)
	\end{array}
	\right], 
	$$
	$$
	A^N=\left[
	\begin{array}{c}
	\sum_{i=1}^{n}\left( \sum_{j=1}^{N} x^*_1(i,j)\mathrm{sgn}\left( z^*(i,j)\right)\right)\\
	\sum_{i=1}^{n}\left( \sum_{j=1}^{N} x^*_2(i,j)\mathrm{sgn}\left( z^*(i,j)\right)\right)\\
	\vdots\\
	\sum_{i=1}^{n}\left( \sum_{j=1}^{N} x^*_n(i,j)\mathrm{sgn}\left( z^*(i,j)\right)\right)
	\end{array} 
	\right]
	\quad \text{and} \quad
	g^N=
	\left[ 
	\begin{array}{c}
	g_1^N\\
	g_2^N\\
	\vdots\\
	g_n^N 
	\end{array}
	\right]. 
	$$
	Moreover, for $s=1,\dots,n$  we denote by $D^N_s$, the matrix obtained from $D^N$ by substituting the $s$-th column by the vector $A^N$. 
	Since $x^*_1,\dots,x_n^* \in B_{\ell_{1}(\Omega_0^n)}$ and $\left|\sum_{i=1}^{n}\left( \sum_{j=1}^{N} x^*_s(i,j)\mathrm{sgn}\left( z^*(i,j)\right)\right) \right| \leq \left\|x^*_s \right\|_{\ell_1(\Omega_0^n)} $ for every $N$, it follows that
	$$
	\lim_{N\rightarrow \infty}\det D^N=1\quad 
	$$
	and, for every $s=1,\dots,n$,
	$$
	 \quad \limsup_{N\rightarrow \infty}\left| \det D^N_s\right| =\limsup_{N\rightarrow \infty}\left|\sum_{i=1}^{n}\left( \sum_{j=1}^{N} x^*_s(i,j)\mathrm{sgn}\left( z^*(i,j)\right)\right) \right|\leq \left\|x^*_s \right\|_{\ell_1}\leq 1.
	$$
	Therefore, it holds $	\limsup_{N\rightarrow \infty}\left|g^N_s \right|\leq  1$
	and it follows that
	$$
	\limsup_{N\rightarrow \infty}\left\|f_N \right\|_{C(\Omega^n)}\leq 1.
	$$
It is always possible to find a sequence of real numbers $\left\lbrace a_N\right\rbrace$ such that $\lim_{N\rightarrow \infty}a_N=1$ and $\left\|a_N f_N \right\|_{C(\Omega^n)}\leq 1$ for every $N$. We obtain that
	$$
	\left\| \phi(z^*)\right\|_{W^*_{x^*_1,\dots,x^*_n}}\geq \left|\phi(z^*)\left( a_N f_N\right)  \right|.
	$$ 
	Since $z^* \in \ell_1(\Omega_{0}^n)$, letting $N\rightarrow \infty$ in the previous inequality, we get
	$$
	\left\| \phi(z^*)\right\|_{W^*_{x_1^*,\dots,x_n^*}}\geq \left\|z^* \right\|_{\ell_1(\Omega_0^n)}.
	$$
	It follows that $W^*_{x_1^*,\dots,x_n^*}$ is isometric to $\ell_{1}(\Omega_{0}^n)$. Moreover, it is easy to see that, for every $s=1,\dots,n$,
	$$
	e^*_{s,j} \xrightarrow[j\rightarrow \infty]{\sigma(\ell_1(\Omega_0^n),W_{x_1^*,\dots,x_n^*})} x^*_s.
	$$  
	
\end{proof}

By recalling Lemma 2 in \cite{A1992}, the following result is a direct consequence of Theorem \ref{l_1-preduals}.
\begin{corollary}\label{cor1}
	Let $X$ be a Banach space such that $X^*=\ell_1(\Omega_0^n)$ and it holds
	$$
	e^*_{s,j} \xrightarrow[j\rightarrow \infty]{\sigma(\ell_1(\Omega_0^n),X)} x^*_s,
	$$
	for all $s=1,\dots,n$. Then $X=W_{x^*_1,\dots,x^*_n}$.  
\end{corollary}

\begin{corollary}\label{cor2}
	Let $x_1^*,\dots,x^*_n\in B_{\ell_1}$, then there exists a closed subspace $W$ of $C(\Omega^n)$ of codimension $n$ such that
	\begin{enumerate}
		\item $W^*=\ell_1$;
		\item the limit points of the standard basis  $\{e^*_j\}$ of $\ell_1$, with respect to the $\sigma(\ell_1,W)$-topology, are $x_1^*,\dots,x^*_n$. 
	\end{enumerate}  
\end{corollary}

In \cite{Casini-Miglierina-Piasecki2014} we proved that, if the standard basis is $\sigma(\ell_1,X)$-convergent to a fixed element $\alpha$ of $B_{\ell_1}$, then $X$ is isometric to $W_\alpha$. On the other hand, if we consider a more general case where we only know that the $w^*$-derivative set of the standard basis of $\ell_1$ is given by a finite number of fixed elements, we have not the unique model of the predual $X$ of $\ell_1$ (see Remark \ref{not commutative}).

\section{Structure of $\ell_1$-preduals and the weak$^*$ fixed point property in $\ell_1$} \label{sec example}

We begin by providing a sufficient condition for the presence of an isometric copy of a hyperplane $W_{\alpha}$ in an $\ell_1$-predual $X$ in terms of the behaviour of the $\sigma(\ell_1,X)$-limit points of the standard basis in $\ell_1$.
This result also shows among which spaces we will not find the desired example.  

\quad
\begin{proposition}\label{Presence of W_e^*}
	Let $X$ be a Banach space such that $X^*=\ell_1$. If there exists a subsequence $\{e^*_{n_k}\}_{k \in \mathbb{N}}$ of the standard basis $\{e^*_n\}$ of $\ell_1$ and an element $e^*\in \ell_1$ such that 
	$$
	e^*_{n_k} \xrightarrow{\sigma(\ell_1,X)} e^*,
	$$
	and $\mathrm{supp}(e^*) \subseteq \{n_k\}_{k \in \mathbb{N}}$, then $X$ contains a $1$-complemented isometric copy of $W_{\alpha}$ with $\alpha=(e^*(n_1),e^*(n_2),\dots, e^*(n_k),\dots)$.
\end{proposition}
\begin{proof}
	Let $K=\{e^*_{n_k}\}_{k \in \mathbb{N}}\subset X^*$, then $K\cap (-K)=\emptyset$. Moreover, it is easy to see that $Y=\overline{\mathrm{span}(K)}^{\left\|\cdot \right\|_{\ell_1} }$ is such that $\overline{K}^{w^*}\subset Y$. Hence, Lemma 1 in \cite{A1992} implies that $Y$ is $w^*$-closed. 
	
	We are now in a position to apply Theorem 1.1 in \cite{Gasparis 2002} to obtain a $w^*$-continuous and contractive projection $P:X^*\rightarrow Y$.
	
	Moreover, let us consider the map $T$
	defined by
	$$
	T\left( \sum_{k=1}^{\infty}a_ke^*_{n_k}\right) =\sum_{k=1}^{\infty}a_kf^*_k,
	$$
	where $\{f^*_k\}$ is the standard basis of $W^*_{\alpha}=\ell_1$ and $\left\lbrace a_k \right\rbrace\subset \mathbb{R}$ such that $\sum|a_k|<\infty$. By Lemma 2 in \cite{A1992}, $T$ is a $w^*$-continuous isometry from $Y$ onto $ W^*_{\alpha}$.
	Let us denote by $^{\perp}Y$ the annihilator of $Y\subseteq X^{*}$ in $X$, since $P$ and $T$ are $w^*$-continuous, there exist two linear bounded operators
	\[
	S:\nicefrac{X}{^{\perp}Y}\longrightarrow X\quad\mathrm{and}\quad
	R:W_{\alpha}\longrightarrow \nicefrac{X}{^{\perp}Y}
	\]
	such that $S^{*}=P$, $R^{*}=T$. Observe that $S$ is an isometry into and $R$ is an isometry onto. Therefore, the linear
	map $S\circ R:W_{\alpha}\longrightarrow X$ is an isometry into and $(S\circ R)(W_{\alpha})\subseteq X$ is an isometric copy of $W_{\alpha}$.
\end{proof}

\begin{remark}
	Proposition \ref{Presence of W_e^*} implies that if the standard basis of $X^*=\ell_1$ has a $w^*$-limit point $e^*$ with a finite support, then $X$ contains an isometric copy of $W_{\alpha}$ with $\alpha = e^*$.  
\end{remark}

We shall need one more result.

\begin{lemma}\label{lemmaDS} (\cite{Dunford Schwartz}, p. 441)
	Let $X$ be a closed subspace of the space $C(K)$ of all real continuous functions on a compact Hausdorff space $K$. For each $q \in K$ let $x_q^*\in X^*$ be defined by
	$$
	x^*_q(f)=f(q),\quad f\in X.
	$$
	Then every extreme point of the closed unit ball of $X^*$ is of the form $\pm x^*_q$ with $q \in K$.
\end{lemma}

\quad

We can now provide the promised example that allows us to answer the open question posed in \cite{Casini-Miglierina-Piasecki2015}.

\begin{example} \label{Example}Let $x_1^*$, $x_2^* \in \ell_1(\Omega_0^2)$ be defined by
\begin{equation*}
	x_1^*(i,j) =
	\begin{cases}
	\frac{1}{2^{2j-1}} & \quad\text{for } i=1 \text{ and } j\ge 1\\
	\frac{1}{2^{2j}} & \quad\text{for } i=2 \text{ and } j\ge 1,
	\end{cases}
\end{equation*}
and $x_2^*(i,j)=0$ for $i=1,2$ and for every $j\geq 1$. Then
$$
W_{x_1^*,x_2^*}=\left\lbrace f\in C(\Omega^2):f(1,0)= \sum_{j=1}^{\infty} \frac{f\left( 1,\frac{1}{j}\right) }{2^{2j-1}} + \sum_{j=1}^{\infty} \frac{f\left( \frac{1}{2},\frac{1}{j}\right) }{2^{2j}} \textrm{ and }  f\left( \frac12,0\right) =0 \right\rbrace.
$$
By Theorem \ref{l_1-preduals}, $W_{x_1^*,x_2^*}^*=\ell_1(\Omega_0^2)$ and
\begin{equation*}
	e^*_{1,j} \xrightarrow {\sigma(\ell_1(\Omega_0^2),W_{x_1^*,x_2^*})} x_1^*=\left(\frac{1}{2},\frac{1}{8},\frac{1}{32},\dots,\frac{1}{4},\frac{1}{16},\frac{1}{64},\dots \right) ,
\end{equation*}
and
\begin{equation*}
	e^*_{2,j} \xrightarrow {\sigma(\ell_1(\Omega_0^2),W_{x_1^*,x_2^*})} x_2^*=\left(0,0,0,\dots,0,0,0,\dots \right).
\end{equation*}
By Theorem \ref{OGT}, $\ell_1(\Omega_0^2)$ lacks the $\sigma(\ell_1(\Omega_0^2),W_{x_1^*,x_2^*})$-FPP. We shall prove that for every $\alpha\in \ell_1$ with $\left\| \alpha\right\| =1$, $W_{\alpha} \not\subset W_{x_1^*,x_2^*}$. This in particular shows that $W_{x_1^*,x_2^*}$ does not have a subspace linearly isometric to a ``bad'' $W_{\alpha}$.

Suppose that there exists $\alpha\in \ell_1$ with $\left\| \alpha\right\| =1$ such that $W_{\alpha} \subset W_{x_1^*,x_2^*}$. Since $W_{\alpha} \subset C(\Omega^2)$, by Lemma \ref{lemmaDS}, Theorem \ref{l_1-preduals} and the fact that $e_n^* \xrightarrow {\sigma(\ell_1,W_{\alpha})} \alpha$, we conclude that there exists a sequence $\left\lbrace \left( 1, \frac{1}{j_n}\right)\right\rbrace _{n \in \mathbb{N}} \subset \Omega_0^2 $ and a sequence of signs $\left\lbrace \epsilon_n\right\rbrace  _{n \in \mathbb{N}}$ such that $\epsilon_n e_{1,j_n}^*$ is an extension of $e_n^*$ to the whole space $W_{x_1^*,x_2^*}$ for all $n \in \mathbb{N}$, except for, at most, a finite number of $n$'s. Without loss of generality for our further considerations, we may assume that it holds for all $n \in \mathbb{N}$. Observe also that the entire sequence $\left\lbrace \epsilon_n e^*_{1,j_n}\right\rbrace_{n\in \mathbb{N}} $ must be $\sigma(\ell_1(\Omega_0^2),W_{x_1^*,x_2^*})$-convergent to $x_1^*$ or $-x_1^*$ (otherwise we would get a contradiction with the $\sigma(\ell_1,W_{\alpha})$-convergence of $\left\lbrace e_n^*\right\rbrace $ to $\alpha$). We consider the first case. The second is similar. So, suppose that $$\epsilon_n e^*_{1,j_n} \xrightarrow {\sigma(\ell_1(\Omega_0^2),W_{x_1^*,x_2^*})} x_1^*.$$ Then, $\epsilon(n)=1$ for all $n \in \mathbb{N}$, except for a finite number of $n$'s. W.L.O.G. we may assume that $\epsilon(n)=1$ for all $n \in \mathbb{N}$. Clearly, $x_1^*$ is an extension of $\alpha$ to the whole space $W_{x_1^*,x_2^*}$.

Let $$A:=\left\lbrace n \in \mathbb{N}: \alpha(n)<0 \right\rbrace .$$ Now we have to consider two cases.

\textsc{Case 1.} The set $A$ is finite. Then there exists $f \in S_{W_{\alpha}}$ such that $\alpha(f)=1$. However, there is no $f \in W_{x_1^*,x_2^*}$ such that $x_1^*(f)=1$ because the point $(1,1,1,\dots,1,1,1,\dots) \notin W_{x_1^*,x_2^*}$. A contradiction.

\textsc{Case 2.} The set $A$ is infinite. Then, for every $n \in A$ we have 
$$\left\|e_n^*-\alpha \right\| = 2.$$ 
On the other hand, for all $n \in \mathbb{N}$
$$\left\| e^*_{1,j_n}-x_1^*\right\| <2.$$
Since $e^*_{1,j_n}-x_1^*$ is an extension of $e_n^*-\alpha$ for all $n \in A$ except for, at most, a finite number of $n$'s that we removed before, we get a contradiction.

Therefore, there is no $\alpha\in \ell_1$ with $\left\| \alpha\right\| =1$ such that $W_{\alpha} \subset W_{x_1^*,x_2^*}$.

Actually, the above considerations provide more information. Namely, by following the reasoning presented for $W_{\alpha} \subset W_{x_1^*,x_2^*}$ and based on Lemma \ref{lemmaDS} and Theorem \ref{l_1-preduals}, we conclude that if $X \subset W_{x_1^*,x_2^*}$ is an $\ell_1$-predual, then $(\mathrm{ext }\,B_{X^*})'$ contains at most three elements (otherwise we would get a contradiction with the fact that $(\mathrm{ext }\,B_{W_{x_1^*,x_2^*}^*})'$ contains exactly three elements). Clearly, if $(\mathrm{ext }\,B_{X^*})'$ has an odd number of elements, then it must contain the origin. In particular, if $(\mathrm{ext }\,B_{X^*})'$ is a singleton, then $X$ is isometric to $c_0$. Finally, if $(\mathrm{ext }\,B_{X^*})'$ has exactly two elements, then $X$ must be isometric to some hyperplane $W_{\alpha}$ and, by what we have already proved, $(\mathrm{ext }\,B_{W_{\alpha}^*})'=\left\lbrace \pm \alpha\right\rbrace $ with $0<\left\| \alpha \right\| <1$.
 
Reassuming, the space $W_{x_1^*,x_2^*}$ does not contain an isometric copy of any $\ell_1$-predual $X$ such that all the elements of $(\mathrm{ext }\,B_{X^*})'$ have norm $1$. 

\end{example}

\begin{remark}\label{not commutative}
Let $W_{x_1^*,x^*_2}$ be as in Example \ref{Example}. Then 
$$W_{x_2^*,x_1^*}=\left\lbrace f\in C(\Omega^2):f(1,0) =0 \textrm{ and } f\left( \frac12,0\right) = \sum_{j=1}^{\infty} \frac{f\left( 1,\frac{1}{j}\right) }{2^{2j-1}} + \sum_{j=1}^{\infty} \frac{f\left( \frac{1}{2},\frac{1}{j}\right) }{2^{2j}} \right\rbrace.
$$
By Theorem \ref{l_1-preduals}, $W_{x_2^*,x_1^*}^*=\ell_1(\Omega_0^2)$ and
\begin{equation*}
	e^*_{1,j} \xrightarrow {\sigma(\ell_1(\Omega_0^2),W_{x_2^*,x_1^*})} x_2^*=\left(0,0,0,\dots,0,0,0,\dots \right),
\end{equation*}
and
\begin{equation*}
	e^*_{2,j} \xrightarrow {\sigma(\ell_1(\Omega_0^2),W_{x_2^*,x_1^*})} x_1^*=\left(\frac{1}{2},\frac{1}{8},\frac{1}{32},\dots,\frac{1}{4},\frac{1}{16},\frac{1}{64},\dots \right) .
\end{equation*}
Hence, $(\mathrm{ext }\,B_{W_{x_1^*,x_2^*}^*})'=(\mathrm{ext }\,B_{W_{x_2^*,x_1^*}^*})'$. However, $W_{x_1^*,x_2^*}$ is not isometric to $W_{x_2^*,x_1^*}$. Suppose, contrary to our claim, that there exists an isometrical isomorphism $T$ from $W_{x_1^*,x_2^*}$ onto $W_{x_2^*,x_1^*}$. Then its adjoint $T^*$ is a weak$^*$ continuous isometry from $\ell_1(\Omega_0^2)=W_{x_2^*,x_1^*}^*$ onto $\ell_1(\Omega_0^2)=W_{x_1^*,x_2^*}^*$. Therefore, there is a permutation $\pi: \left\lbrace (i,j):i\in \left\lbrace 1,2\right\rbrace  \textrm{, } j \in \mathbb{N} \right\rbrace  \to \left\lbrace (i,j):i\in \left\lbrace 1,2\right\rbrace  \textrm{, } j \in \mathbb{N} \right\rbrace$ and a function $\epsilon:\left\lbrace (i,j):i\in \left\lbrace 1,2\right\rbrace  \textrm{, } j \in \mathbb{N} \right\rbrace \to \left\lbrace -1,1\right\rbrace $ such that
$$T^*(e^*_{i,j})=\epsilon(i,j)e^*_{\pi(i,j)}$$
for every $i=1,2$ and every $j\in \mathbb{N}$. Moreover, since the sequence $\left\lbrace e^*_{2,j}\right\rbrace_{j \in \mathbb{N}} $ is $\sigma(\ell_1(\Omega_0^2),W_{x_2^*,x_1^*})$-convergent to $x_1^*$, the sequence $\left\lbrace T^*(e^*_{2,j})\right\rbrace _{j\in\mathbb{N}}$ is $\sigma(\ell_1(\Omega_0^2),W_{x_1^*,x_2^*})$-convergent to $T^*(x_1^*)= \pm x_1^*$. W.L.O.G. we may assume that $T^*(x_1^*)= x_1^*$. Then there is a sequence $\left\lbrace \left( 1,\frac{1}{n_j}\right)  \right\rbrace _{j \in \mathbb{N}} \subset \Omega_0^2$ such that $T^*(e^*_{2,j})=e^*_{1,n_j}$ for all $j \in \mathbb{N}$ except for, at most, a finite number of $j$'s. We have $$\left\|T^*( e^*_{2,j}- x_1^*)  \right\| =\left\|e^*_{1,n_j}- x_1^* \right\| = 2 - \frac{1}{2^{2n_{j}-1}}$$
whereas
$$\left\|e^*_{2,j}- x_1^* \right\| = 2 - \frac{1}{2^{2j}}.$$
This contradicts our assumption that $T^*$ is an isometry.
\end{remark}

  \section{Final Remarks}
 To conclude this paper it is interesting to give a look to the landscape surrounding the results presented here. Indeed, some remarks will clarify the main motivations of the present paper.
 
 As described in the introduction, the problem of finding necessary and sufficient conditions for FPP is usually very hard. In this vein, a milestone in the study of FPP is the result asserting that a closed subspace of $L_1([0,1])$ enjoys the FPP if and only if it is reflexive (\cite{Maurey1980-1981}, \cite{Dowling-Lennard1997}). Moreover, closely related to this topic is the old and well-known open problem to prove (or disprove) if any reflexive Banach space satisfies the FPP (see, e.g., \cite{Dominguez-Japon} and the reference therein).
 
 We focused our study on the $w^*$-FPP in the duals of separable Lindenstrauss spaces. Besides the characterization of the $w^*$-FPP that we mentioned in Theorem \ref{OGT}, we obtained also another characterization by considering spaces $A(S)$ of affine continuous functions on a Choquet simplex $S$. Namely, in  \cite{Casini-Miglierina-PiaseckiPAMS} we proved the following result.
 \begin{thm}\label{Characterization A(S)}	
 	Let $X$ be a predual of $\ell_1$. The following statements are equivalent:
 	\begin{enumerate}
 		\item \label{item noFPP1} $\ell_1$ lacks the $\sigma(\ell_1,X)$-FPP;
 		
 		\item \label{item affine}there is a quotient of $X$ isometric to some infinite-dimensional $A(S)$ space;
 		
 		\item \label{item containement}there is a quotient of $X$ containing an isometric copy of some infinite-dimensional $A(S)$ space.
 		\end{enumerate}
 \end{thm}
Also here, it is worth mentioning that the space $A(S)$ in statements (2) and (3) of Theorem \ref{Characterization A(S)} cannot be replaced by any $C(K)$ space (see Remark 3.5 in \cite{Casini-Miglierina-PiaseckiPAMS}).

As for the characterization of the $\sigma(\ell_1,X)$-FPP given by a ``bad'' $W_\alpha$, the novelty of this result lies in the presence of particular quotients in the predual space $X$. Since sufficient conditions for the lack of the $\sigma(X^*,X)$-FPP, where $X$ is a separable Banach space, were obtained by requiring the presence in the predual $X$ a subspace isometric to a ``bad'' $W_\alpha$ (Theorem 3.7 in \cite{Casini-Miglierina-Piasecki2015}) or an infinite-dimensional space $A(S)$ (Theorem 2.3 in \cite{Casini-Miglierina-PiaseckiPAMS}), it is natural to ask if, in the setting of $\ell_1$-preduals, also necessary conditions can be obtained by supposing that $X$ contains an isometric copy of a ``bad'' $W_\alpha$ or an infinite-dimensional space $A(S)$. In \cite{Casini-Miglierina-PiaseckiPAMS}, we succeeded to show that considering the quotients in Theorem \ref{Characterization A(S)} is unavoidable when the characterization of lack of the $\sigma(\ell_1,X)$-FPP is given in terms of a space $A(S)$. Indeed, Example 3.3 in \cite{Casini-Miglierina-PiaseckiPAMS} exhibits a predual $X$ of $\ell_1$ such that $\ell_1$ lacks the $\sigma(\ell_1,X)$-FPP, but $X$ does not contain an isometric copy of any infinite-dimensional space $A(S)$. Unfortunately, this space $X$ is itself a ``bad'' $W_\alpha$. Therefore, it remained an open problem if quotients are strictly necessary in the characterization presented in Theorem \ref{OGT}, obtained by means of a ``bad'' $W_\alpha$. Example \ref{Example} in this paper answers this question by closing the whole discussion. Indeed, the space $\ell_1$ lacks the $\sigma(\ell_1,W_{x_1^*,x_2^*})$-FPP, but $W_{x_1^*,x_2^*}$ does not contain an isometric copy of any  ``bad'' $W_\alpha$. We also recall that for any infinite-dimensional space $A(S)$, every element of $(\mathrm{ext }\,B_{A(S)^*})'$ has norm $1$ (see, e.g., the proof of Theorem 2.3 in \cite{Casini-Miglierina-PiaseckiPAMS}). Therefore, by the summary comment in Example \ref{Example}, the space $W_{x_1^*,x_2^*}$ does not contain an isometric copy of any infinite-dimensional space $A(S)$.


\begin{thebibliography}{1}



\bibitem{A1992} D. E. Alspach, {\it A $\ell_{1}$-predual which is not
isometric to a quotient of $C(\alpha)$}, arXiv:math/9204215v1
{[}math.FA{]} 27 Apr 1992.




\bibitem{Casini-Miglierina-Piasecki2014} E. Casini, E. Miglierina,
\L. Piasecki, {\it Hyperplanes in the space of convergent sequences and
preduals of $\ell_1$}, Canad. Math. Bull. 58 (2015), 459-470.

\bibitem{Casini-Miglierina-Piasecki2015} E. Casini, E. Miglierina, \L. Piasecki, {\it Separable Lindenstrauss spaces whose duals lack the weak$^*$ fixed point property for nonexpansive mappings}, Studia Math. 238 (2017), 1-16.

\bibitem{Casini-Miglierina-Piasecki-Vesely2016} E. Casini, E. Miglierina,
\L. Piasecki and L. Vesel\'{y}, {\it Rethinking polyhedrality for Lindenstrauss spaces}, Israel J. Math. 216 (2016), 355-369.



\bibitem{Casini-Miglierina-PiaseckiPAMS} E. Casini, E. Miglierina, \L. Piasecki, {\it Weak$^*$ fixed point property and the space of affine functions}, Proc. Amer. Math. Soc. 149 (2021), 1613-1620.

\bibitem{Dominguez-Japon} T. Dom\'{\i}nguez Benavides, M. A. Jap\'{o}n,
{\it Fixed point properties and reflexivity in variable Lebesgue spaces}, 
J. Funct. Anal. 280 (2021), Paper No. 108896, 22 pp. 

\bibitem{Dowling-Lennard1997} P. N. Dowling, C. J. Lennard, {\it Every
	nonreflexive subspace of $L_1[0,1]$ fails the fixed point property},
Proc. Amer. Math. Soc. 125 (1997), 443-446.

\bibitem{Dunford Schwartz}N. Dunford and J.  T. Schwartz, {\it Linear operators.
	Part I. General theory}, reprint of the 1958 original, John Wiley \&
Sons, Inc., New York, 1988. 

\bibitem{Durier & Papini} R. Durier, P. L. Papini, {\it Polyhedral
	norms in an infinite dimensional space}, Rocky Mountain J.
Math. 23 (1993), 623-645.

\bibitem{Fonf Lindenstrauss & Phelps 2001}V. P. Fonf, J. Lindenstrauss, R. R. Phelps, {\it Infinite Dimensional Convexity}, in Handbook
of the Geometry of Banach Spaces 1 (eds.: W. B. Johnson
and J. Lindenstrauss), Elsevier Science, 2001, 599-670.

\bibitem{Fonf & Vesely} V. P. Fonf, L. Vesel\'y, {\it Infinite-Dimensional
	Polyhedrality}, Canad. J. Math. 56 (2004), 472-494.



\bibitem{Gasparis 2002} I. Gasparis, {\it On contractively complemented
subspaces of $L_{1}$-preduals}, Israel J. Math. 123
(2002), 77-92.

\bibitem{Go1} K. Goebel, {\it Concise course on fixed point theorems}, Yokohama Publishers (2002).

\bibitem{GoKi} K. Goebel, W. A. Kirk, {\it Topics in metric fixed point theory}, Cambridge University Press (1990).

\bibitem{Karlovitz1976} L. A. Karlovitz, {\it On nonexpansive mappings}, Proc. Amer. Math. Soc. 55
(1976), 321-325.


\bibitem{Lacey book} H. E. Lacey, {\it The isometric theory of classical
	Banach spaces}, Die Grundlehren der mathematischen
Wissenschaften 208, Springer-Verlag, New York-Heidelberg,
1974.

\bibitem{Lazar1969}A. J. Lazar, {\it Polyhedral Banach spaces and extensions
	of compact operators}, Israel J. Math. 7 (1969),
357-364.



\bibitem{Lazar-Lindenstrauss1971} A. J. Lazar and J. Lindenstrauss,
{\it Banach spaces whose duals are $L_1$-spaces and their representing
matrices}, Acta Math. 126 (1971), 165-194.





\bibitem{Lindenstrauss 1964}J. Lindenstrauss, {\it Extensions of compact
operators}, Mem. Amer. Math. Soc. 48 (1964).


\bibitem{Lindenstrauss-Wulbert} J. Lindenstrauss, D. E. Wulbert, {\it On the classification of the Banach spaces whose duals are $L_1$ spaces}, J. Funct. Anal. 4 (1969), 332-349.



\bibitem{Maurey1980-1981} B. Maurey, {\it Points fixes des contractions de
	certain faiblement compacts de $L^1$}, in: Seminaire d'Analyse
Fonctionelle (Paris), Expos\'{e} \textbf{VIII}, \'{E}cole
Polytechnique, Centre de Mathematiques, Paris, 1980-1981.


\bibitem{Semadeni1964} Z. Semadeni, {\it Free compact convex sets}, Bull. Acad.
Polon. Sci. S\'{e}r Sci. Math. Astronom. Phys. 13 (1964),
141-146.


\bibitem{Zippin1969} M. Zippin, {\it On some subspaces of Banach spaces
whose duals are $L_{1}$ spaces}, Proc. Amer. Math. Soc. 23 (1969), 378-385.






\bibitem{Zippin 2018} M. Zippin, {\it Correction to ``On some subspaces of Banach spaces
	whose duals are $L_{1}$ spaces''}, Proc. Amer. Math Soc. 146 (2018), 5257-5262.

\end{thebibliography}
\end{document}